\documentclass[11pt]{article}
\usepackage[margin=1.1in]{geometry}
\usepackage{amsmath,amssymb,amsthm}
\usepackage{booktabs}
\usepackage{multirow}
\usepackage{colortbl}
\usepackage{microtype}
\usepackage{tikz}
\usepackage{float}
\usepackage{caption}
\usepackage[colorlinks,citecolor=blue,linkcolor=blue,urlcolor=blue]{hyperref}
\usepackage{flafter}
\usepackage{placeins}

\definecolor{jhull}{RGB}{143,182,217}
\definecolor{jlight}{RGB}{136,199,231}
\definecolor{jdark}{RGB}{69,101,117}

\theoremstyle{plain}
\newtheorem{theorem}{Theorem}
\newtheorem{proposition}[theorem]{Proposition}
\newtheorem{lemma}[theorem]{Lemma}
\newtheorem{corollary}[theorem]{Corollary}
\theoremstyle{definition}
\theoremstyle{plain}

\newcolumntype{S}{>{\columncolor{black!7}[0pt][0pt]}c}

\newenvironment{romanlist}
  {\begin{enumerate}}
  {\end{enumerate}}

\newlength{\treeindent}
\newlength{\treebar}
\newcommand{\treemid}{\hspace{\treeindent}\makebox[1.3em][l]{%
  \rule[-0.3\normalbaselineskip]{0.4pt}{\normalbaselineskip}%
  \rule[\treebar]{0.6em}{0.4pt}}}
\newcommand{\treelast}{\hspace{\treeindent}\makebox[1.3em][l]{%
  \rule[\treebar]{0.4pt}{\dimexpr0.7\normalbaselineskip-\treebar\relax}%
  \rule[\treebar]{0.6em}{0.4pt}}}
\newcommand{\treepass}{\hspace{\treeindent}\makebox[1.3em][l]{%
  \rule[-0.3\normalbaselineskip]{0.4pt}{\normalbaselineskip}}}
\newcommand{\treeblank}{\hspace{\treeindent}\hspace{1.3em}}
\newcommand{\treetop}{\hspace{\treeindent}\makebox[1.3em][l]{%
  \rule[-0.3\normalbaselineskip]{0.4pt}{\dimexpr0.3\normalbaselineskip+\treebar\relax}%
  \rule[\treebar]{0.6em}{0.4pt}}}
\newcommand{\treelastto}{\hspace{\treeindent}\makebox[1.3em][l]{%
  \rule[\treebar]{0.4pt}{\dimexpr0.7\normalbaselineskip-\treebar\relax}%
  \rule[\treebar]{\dimexpr1.3em+\treeindent\relax}{0.4pt}}}
\newcommand{\treejoin}{\hspace{\treeindent}\makebox[1.3em][l]{%
  \rule[\treebar]{0.4pt}{\dimexpr0.7\normalbaselineskip-\treebar\relax}%
  \rule[\treebar]{0.6em}{0.4pt}}}

\newcommand{\Rn}{\mathbb{R}}
\newcommand{\Sph}{\mathbb{S}}
\newcommand{\conv}{\operatorname{conv}}
\newcommand{\Aut}{\operatorname{Aut}}

\title{Non-convex polytopes realizing\\
kissing contact graphs in dimensions 5--7}
\author{Matthew Self}
\date{October 15, 2026}

\begin{document}
\maketitle
\begingroup
\renewcommand{\thefootnote}{}
\begin{NoHyper}
\footnotetext{Emerald Hills, California, USA.
  Corresponding author: \texttt{matthew@mself.com}.
  ORCID 0009-0003-1103-3967.}
\end{NoHyper}
\endgroup

\begin{abstract}
Each of the 12 known conjecturally optimal kissing configurations in
dimensions 5, 6, and 7 carries a \emph{contact-graph polytope}: a polytope
whose vertices are the configuration points and whose edges are exactly
its contacts.
Eight of these polytopes are non-convex and distinct from their convex hulls, which include longer edges.
Each polytope is radial about the sphere's center: its boundary meets every ray from the center exactly once.
It is the unique radial polytope with those vertices and edges.
Seven of the nine non-lattice configurations carry polytopes constructed
from the lattice contact polytope in the same dimension by splitting some
facets and reassembling the pieces. This split-and-reassemble construction
describes the configurations by facets rather than by layers.
The fold angles of the reassembled facets explain the inner products $1/(10-n)$,
giving $1/5$, $1/4$, and $1/3$ in dimensions 5, 6, and 7, respectively.
The contact graphs of all 12 configurations are infinitesimally rigid.
We also prove that $Q_5$ and $R_5$ are infinitesimally jammed, extending the 10 previously established cases to all 12 configurations.
All mathematical computations use exact arithmetic.
\end{abstract}
 


\FloatBarrier

\section{Introduction}\label{sec:intro}

A kissing configuration specifies points on a sphere and a contact graph joining pairs at unit distance. We ask whether these points and segments are the vertices and edges of a polytope. We call such a polytope a \emph{contact-graph polytope} of the configuration. The related lattice term \emph{contact polytope} denotes the convex hull of the shortest vectors. The convex hull need not realize the contact graph, since it can include longer edges in addition to the contacts. In that case, any contact-graph polytope must be non-convex.

In this paper, we show that each of the 12 known conjecturally optimal kissing configurations in dimensions 5, 6, and 7 carries a contact-graph polytope. Eight of these polytopes are non-convex and distinct from their convex hulls. Each is unique among radial contact-graph polytopes, whose boundaries meet every ray from the sphere's center exactly once.

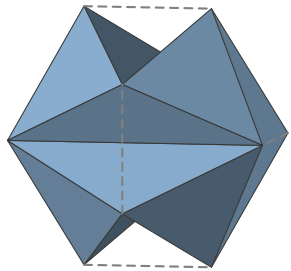
\begin{figure}
\centering
\begin{tikzpicture}[scale=1.0,line join=round,line cap=round]
  \filldraw[fill=jlight!90!jdark,draw=black!78,line width=0.4pt] (-1.518,0.115) -- (-0.843,-1.695) -- (-0.843,1.728) -- cycle;
  \filldraw[fill=jlight!0!jdark,draw=black!78,line width=0.4pt] (-0.843,-1.695) -- (0.337,-0.691) -- (-0.337,-1.020) -- cycle;
  \filldraw[fill=jlight!43!jdark,draw=black!78,line width=0.4pt] (0.843,-1.728) -- (0.843,1.695) -- (1.855,0.049) -- cycle;
  \filldraw[fill=jlight!6!jdark,draw=black!78,line width=0.4pt] (-0.843,1.728) -- (0.337,1.020) -- (-0.337,0.691) -- cycle;
  \filldraw[fill=jlight!49!jdark,draw=black!78,line width=0.4pt] (-1.855,-0.049) -- (-0.843,-1.695) -- (-0.337,-1.020) -- cycle;
  \filldraw[fill=jlight!100!jdark,draw=black!78,line width=0.4pt] (-1.855,-0.049) -- (-0.843,1.728) -- (-0.337,0.691) -- cycle;
  \filldraw[fill=jlight!12!jdark,draw=black!78,line width=0.4pt] (-0.337,-1.020) -- (0.843,-1.728) -- (1.518,-0.115) -- cycle;
  \filldraw[fill=jlight!98!jdark,draw=black!78,line width=0.4pt] (-1.855,-0.049) -- (-0.337,-1.020) -- (1.518,-0.115) -- cycle;
  \filldraw[fill=jlight!61!jdark,draw=black!78,line width=0.4pt] (-0.337,0.691) -- (0.843,1.695) -- (1.518,-0.115) -- cycle;
  \filldraw[fill=jlight!37!jdark,draw=black!78,line width=0.4pt] (-1.855,-0.049) -- (-0.337,0.691) -- (1.518,-0.115) -- cycle;
  \draw[draw=black!78,line width=0.8pt,dash pattern=on 3pt off 2.4pt] (-0.843,-1.695) -- (0.843,-1.728);
  \draw[draw=black!78,line width=0.8pt,dash pattern=on 3pt off 2.4pt] (-0.843,1.728) -- (0.843,1.695);
  \draw[draw=black!78,line width=0.8pt,dash pattern=on 3pt off 2.4pt] (-0.337,-1.020) -- (-0.337,0.691);
  \draw[draw=black!78,line width=0.8pt,dash pattern=on 3pt off 2.4pt] (1.855,0.049) -- (1.518,-0.115);
\end{tikzpicture}
\caption{Jessen's icosahedron overlaid with the edges of its convex hull
(dashed)}
\label{fig:jessen}
\end{figure}

Jessen's icosahedron~\cite{Jessen1967}, shown in Figure~\ref{fig:jessen},
illustrates how a non-convex polytope can share all its vertices with its
convex hull. Its hull is an irregular icosahedron with the same 12 vertices.
Its 12 isosceles faces fold inward, forming six dents along their long edges.
The non-convex contact-graph polytopes in dimensions 5, 6, and 7 have similar dents.

The analogy is imperfect in an important respect. Jessen's icosahedron has edges that are
\emph{longer} than any in its convex hull. In our case, it is the \emph{convex hull}
that has edges longer than those of the non-convex contact-graph polytopes.
The non-convex contact-graph polytopes are also
\emph{simpler} than their convex hulls: they have fewer edges, facets,
and facet types.

Seven of the nine non-lattice configurations carry polytopes constructed
by splitting and reassembling facets of the lattice contact polytope in
the same dimension.
The facets are detached, some are split at vertex figures, and the
resulting pieces become facets of the target polytope.
Every piece is used once (Theorem~\ref{thm:parts}).
The allowed fold angles of the reassembled facets explain the inner products
$1/(10-n)$, which do not occur in the lattice configurations
(Section~\ref{sec:splits}).
This is a geometric construction that is
distinct from the constructions by layers, in which a configuration is
built by stacking and coloring layers of a lower-dimensional one, as in
Cohn, Jiao, Kumar, and Torquato's enumeration in dimensions 5, 6,
and 7~\cite{CJKT2011} and Cohn and Rajagopal's analysis of $Q_5$ and
$R_5$~\cite{CohnR2024}.

The edges of each contact-graph polytope form a bar framework in
Euclidean space. We verify infinitesimal rigidity of all 12 frameworks
(Section~\ref{sec:rigid}). The same contacts define the spherical
jamming problem, in which the vertices remain on the sphere and contact
distances may increase but not decrease. Cohn, Jiao, Kumar, and
Torquato~\cite{CJKT2011} verified infinitesimal jamming for the 10
configurations they list.
We extend this result to $Q_5$ and $R_5$, bringing the total to 12
and answering Sz\"{o}ll\H{o}si's question about $Q_5$~\cite{Szollosi2023}.

The 12 configurations are the lattice ones $D_5$, $E_6$, and
$E_7$~\cite{ConwaySloane1999}, together with nine others:
$L_5$~\cite{Leech1967}, $Q_5$~\cite{AndreanovKallus2020,Szollosi2023}, and
$R_5$~\cite{CohnR2024} in dimension 5, and
$C6\text{-}72b/d/c$ and $C7\text{-}126c$~\cite{ConwaySloane1995},
together with $C7\text{-}126b/d$~\cite{CJKT2011}, in dimensions 6 and 7.
The labels in dimensions 6 and 7 are those of Cohn, Jiao, Kumar, and
Torquato~\cite{CJKT2011}. The first non-lattice
packings in dimensions 6 and 7 are Leech's~\cite{Leech1969}.

The kissing number in these three dimensions is known only conjecturally:
the bounds are $40$--$44$,
$72$--$77$, and $126$--$134$~\cite{Vallentin2026},
and the conjectured optima are not unique.

\paragraph{Notation and summary.}
For each configuration $V$, its \emph{contact polytope} $P(V)$ is the
convex hull of its points. We write $P^{*}(V)$ for the exhibited
contact-graph polytope on $V$. We omit the argument
when the configuration is clear.
Table~\ref{tab:all} compares the edge, facet, and facet-type counts
of $P^{*}$ and $P$. It also lists dent counts and inner products along
longer hull edges.

\begin{table}[ht]
\centering\small
\setlength{\tabcolsep}{6pt}
\renewcommand{\arraystretch}{1.1}
\begin{tabular}{cclrrrrrrrc}
\toprule
 & & & \multicolumn{4}{c}{$P^{*}$ (contact-graph polytope)}
      & \multicolumn{4}{c}{$P$ (contact polytope)} \\
\cmidrule(lr){4-7}\cmidrule(lr){8-11}
$n$ & $|V|$ & Config.\
   & Edges & Facets & Types & Dents & Edges & Facets & Types
   & Extra $\langle u,v\rangle$ \\
\midrule
\multirow{4}{*}{$5$} & \multirow{4}{*}{$40$}
   & $D_5$ & $240$ & $42$ & $2$ & -- & \multicolumn{4}{S}{$P = P^{*}$} \\
 & & $L_5$ & $240$ & $50$ & $3$ & -- & \multicolumn{4}{S}{$P = P^{*}$} \\
 & & $Q_5$ & $240$ & $62$ & $4$ & $10$ & $250$ & $92$ & $5$ & $1/5$ \\
 & & $R_5$ & $240$ & $70$ & $5$ & $10$ & $250$ & $100$ & $6$ & $1/5$ \\
\midrule
\multirow{4}{*}{$6$} & \multirow{4}{*}{$72$}
   & $E_6$ & $720$ & $54$ & $1$ & -- & \multicolumn{4}{S}{$P = P^{*}$} \\
 & & $C6\text{-}72b$ & $720$ & $86$ & $3$ & $16$ & $736$ & $214$ & $4$ & $1/4$ \\
 & & $C6\text{-}72d$ & $720$ & $118$ & $4$ & $32$ & $752$ & $358$ & $6$ & $1/4$ \\
 & & $C6\text{-}72c$ & $720$ & $150$ & $3$ & $48$ & $768$ & $502$ & $5$ & $1/4$ \\
\midrule
\multirow{4}{*}{$7$} & \multirow{4}{*}{$126$}
   & $E_7$ & $2{,}016$ & $632$ & $2$ & -- & \multicolumn{4}{S}{$P = P^{*}$} \\
 & & $C7\text{-}126c$ & $2{,}016$ & $686$ & $4$ & $27$ & $2{,}043$ & $1{,}334$ & $5$ & $1/3$ \\
 & & $C7\text{-}126b$ & $1{,}984$ & $536$ & $4$ & $64$ & $2{,}176$ & $2{,}376$ & $6$ & $1/4$ \\
 & & $C7\text{-}126d$ & $1{,}984$ & $590$ & $8$ & $91$ & $2{,}187$ & $2{,}822$ & $14$ & $1/4,\,1/3$ \\
\bottomrule
\end{tabular}
\caption{Summary of results. \emph{Types} counts congruence classes of
facets, and \emph{Dents} the reflex ridges of $P^{*}$; $P^{*}$ is convex
exactly when $P = P^{*}$; the extra
$\langle u,v\rangle$ are those realized by edges of $P$ and by no edge of
$P^{*}$.}
\label{tab:all}
\end{table}

\FloatBarrier

\paragraph{Terms.}
\begin{description}
\item[Kissing configuration.] A finite set $V \subset \Sph^{n-1}$ with
  $\langle u,v\rangle \le \tfrac12$ for all distinct $u,v \in V$.
  These points represent the points of tangency of pairwise non-overlapping
  unit spheres with a central unit sphere.
\item[Cell.] A convex polytope of any dimension. Facets, their faces,
  and the pieces of a split are cells. The cells used to assemble a
  polytope on $V$ have vertices in $V$.
\item[Polytope.] In dimension $n$, we allow a polytope to be the closed
  bounded region enclosed by an embedded polyhedral $(n-1)$-sphere
  with finitely many convex facets. The facets meet face-to-face: the
  intersection of two facets is empty or a face of both. Facets sharing
  a ridge must lie on different hyperplanes. This excludes flat ridges
  and prevents a facet from being subdivided into coplanar facets.
  A polytope \emph{on} $V$ has vertex set exactly $V$.
\item[Faces.] A facet is an $(n-1)$-dimensional cell of the boundary.
  Its $(n-2)$-dimensional faces are the \emph{ridges} of the polytope,
  and its $(n-3)$-dimensional faces are the \emph{peaks}.
  The edges are the one-dimensional faces of the facets.
\item[Unit-edge.] A polytope or cell is \emph{unit-edge} when every
  edge has length~$1$.
\item[Contact-graph polytope.] A polytope on a kissing configuration
  $V$ whose edges are exactly the contacts of $V$. Equivalently, its
  $1$-skeleton is the contact graph of $V$. Every contact-graph polytope
  is unit-edge.
\item[$P$ and $P^{*}$.] Each configuration is named by its contact
  polytope: $Q_5$ is
  the convex hull of the 40 points, and $Q_5^{*}$ the polytope on those
  points and their contacts. Every point of $V$ is a vertex of
  $\conv V$, so $\conv P^{*}(V)=P(V)$.
\item[Radial.] A polytope is \emph{radial} about the sphere's center,
  the origin, if every ray from the origin meets its boundary exactly
  once. Equivalently, it is star-shaped about the origin with no facet
  hyperplane containing the origin. Radial projection
  $x\mapsto x/\|x\|$ then maps its boundary homeomorphically onto
  $\Sph^{n-1}$.
\item[External and internal hyperplanes.] A hyperplane is \emph{external}
  when all of $V$ lies weakly on one side of it, and \emph{internal}
  otherwise. A facet is external or internal with its hyperplane.
\item[Reflex ridge.] A ridge whose interior dihedral angle exceeds $\pi$.
\item[Dent.] The two facets meeting at a reflex ridge. This terminology
  is inspired by Connelly and Schlenker's denting construction for
  polyhedra~\cite{ConnellySchlenker2010}.
\item[Vertex-figure split.] For a unit-edge cell with vertices in $V$,
  we use the convex hull of the neighbors of a vertex $v$ as its
  \emph{vertex figure}. These neighbors lie in the hyperplane
  $\langle x,v\rangle=\tfrac12$.
  A \emph{vertex-figure split}, or \emph{split} for short, cuts the
  cell along this hyperplane at a vertex not joined to every other
  vertex. The two pieces are a pyramid over the vertex figure and a
  diminished cell, which is the rest. When the cutting hyperplane
  passes through the cell's circumcenter we call the split \emph{equatorial}.
  The vertex figure is a facet of each piece.

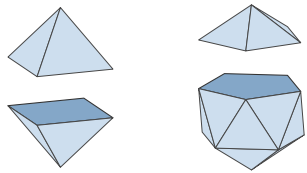
\begin{figure}[ht]
\centering
\begin{tikzpicture}[scale=1.075,line join=round,line cap=round,
  face/.style={fill=jhull!45!white,draw=black!72,line width=0.3pt},
  vf/.style={fill=jhull!92!black,draw=black!80,line width=0.4pt}]
\begin{scope}[shift={(0.000,0)}]
  \filldraw[vf] (0.6460,-0.0744) -- (0.2876,0.1672) -- (-0.6460,0.0744) -- (-0.2876,-0.1672) -- cycle;
  \filldraw[face] (0.0000,-0.6830) -- (-0.6460,0.0744) -- (-0.2876,-0.1672) -- cycle;
  \filldraw[face] (0.0000,-0.6830) -- (-0.2876,-0.1672) -- (0.6460,-0.0744) -- cycle;
  \filldraw[face] (0.0000,1.2819) -- (-0.6460,0.6733) -- (-0.2876,0.4317) -- cycle;
  \filldraw[face] (0.0000,1.2819) -- (-0.2876,0.4317) -- (0.6460,0.5244) -- cycle;
\end{scope}
\begin{scope}[shift={(2.350,0)}]
  \filldraw[face] (0.0000,-0.7165) -- (0.3318,-0.4692) -- (0.6490,-0.2847) -- cycle;
  \filldraw[vf] (0.6061,0.2506) -- (0.4440,0.4481) -- (-0.3318,0.4692) -- (-0.6490,0.2847) -- (-0.0694,0.1497) -- cycle;
  \filldraw[face] (-0.6061,-0.2506) -- (-0.4440,-0.4481) -- (-0.6490,0.2847) -- cycle;
  \filldraw[face] (0.3318,-0.4692) -- (0.6490,-0.2847) -- (0.6061,0.2506) -- cycle;
  \filldraw[face] (0.0000,-0.7165) -- (-0.4440,-0.4481) -- (0.3318,-0.4692) -- cycle;
  \filldraw[face] (-0.6490,0.2847) -- (-0.4440,-0.4481) -- (-0.0694,0.1497) -- cycle;
  \filldraw[face] (-0.0694,0.1497) -- (0.3318,-0.4692) -- (0.6061,0.2506) -- cycle;
  \filldraw[face] (-0.4440,-0.4481) -- (0.3318,-0.4692) -- (-0.0694,0.1497) -- cycle;
  \filldraw[face] (0.0000,1.3154) -- (0.6061,0.8495) -- (0.4440,1.0469) -- cycle;
  \filldraw[face] (0.0000,1.3154) -- (-0.6490,0.8836) -- (-0.0694,0.7485) -- cycle;
  \filldraw[face] (0.0000,1.3154) -- (-0.0694,0.7485) -- (0.6061,0.8495) -- cycle;
\end{scope}
\end{tikzpicture}
\caption{Equatorial and non-equatorial vertex-figure splits in three dimensions}
\label{fig:split3d}
\end{figure}

\end{description}

\paragraph{Main results.}

\begin{theorem}[Existence]\label{thm:exist}
Each of the 12 configurations in Table~\ref{tab:all} carries a contact-graph
polytope on its vertex set $V$, denoted $P^{*}(V)$.
These polytopes are radial about the sphere's center.
For $D_5$, $L_5$, $E_6$, and $E_7$, this is the convex hull
$P(V)$. For the other eight configurations, $P^{*}(V)$ is non-convex and
distinct from $P(V)$. The external facets of $P^{*}(V)$ are exactly the
unit-edge facets of $P(V)$.
\end{theorem}

The edges of the contact polytopes of $D_5$, $E_6$, and $E_7$ are
exactly the contacts, as follows from their classical
descriptions~\cite{Coxeter1973}. For $L_5$, equality of the hull edges
and contacts is checked from its 50 facets.
The substance of Theorem~\ref{thm:exist} is the eight configurations with
$P^{*} \neq P$.

\begin{theorem}[Uniqueness]\label{thm:unique}
Each of the 12 configurations $V$ carries exactly one radial
contact-graph polytope, namely $P^{*}(V)$.
\end{theorem}

Every symmetry of $V$ carries $P^{*}(V)$ to a radial contact-graph
polytope on $V$, which equals $P^{*}(V)$ by uniqueness.
Hence $\Aut(P^{*}(V))=\Aut(V)$, where $\Aut$ denotes the geometric
symmetry group.

\begin{theorem}[Splits]\label{thm:parts}
Let $P^{*}$ be the contact-graph polytope of one of the non-lattice
configurations $L_5$, $Q_5$, $R_5$,
$C6\text{-}72b$, $C6\text{-}72d$, $C6\text{-}72c$, $C7\text{-}126c$, and let
$P_0$ be the contact polytope of the lattice configuration in the same
dimension. The facets of $P_0$ can be detached, split at vertex figures,
and reassembled by rigid motions to form the boundary of $P^{*}$, using
every piece once. A facet may be left unsplit or split more than once.
The same holds for $C7\text{-}126d^{*}$ with $C7\text{-}126b^{*}$ in place
of $P_0$.
\end{theorem}

\paragraph{Organization.}
Section~\ref{sec:n5} works through the five-dimensional configurations in turn and
records their facet censuses.
Section~\ref{sec:splits} analyzes the geometry of vertex-figure splits and
derives the fold angle.
Sections~\ref{sec:n6} and~\ref{sec:n7} then apply the same procedure to
the six- and seven-dimensional cases.
Section~\ref{sec:rigid} verifies infinitesimal rigidity of all 12 skeletons
and proves infinitesimal jamming for $Q_5$ and $R_5$.
Section~\ref{sec:method} proves the existence, uniqueness, and split theorems
and describes the exact computations.

The vertex counts and facet censuses of the configurations
$D_5$, $E_6$, and $E_7$ are
classical~\cite{Coxeter1973}. All others reported below,
including the convex hulls, are outputs of the procedure of
Section~\ref{sec:method}.

\FloatBarrier

\section{Dimension 5}\label{sec:n5}

\paragraph{\boldmath $L_5$.}
The $D_5$ contact polytope is the rectified $5$-orthoplex, with
10 $16$-cells ($4$-orthoplexes) and 32 rectified $5$-cells as facets.
The 16-cell facets can be split equatorially into two octahedral pyramids and
$L_5$ is obtained by splitting eight of the 10 $16$-cells into 16 octahedral pyramids
and reassembling the pieces. No two octahedral pyramids are joined along
their bases, since this would recreate a $16$-cell.

This construction is analogous to how the facets of the cuboctahedron can be reassembled to form the anticuboctahedron
($J_{27}$~\cite{Johnson1966}), except that some of $D_5$'s facets are split first.
In three dimensions there are no smaller unit-edge facets to split the cuboctahedron's facets into.
The 240 edges of $L_5$ are exactly the contacts of its configuration,
so $L_5^{*}=L_5$.

\begin{table}[!ht]
\centering
\begin{tabular}{lrrrrrr}
\toprule
Facets & $D_5$ & $L_5$ & $Q_5^{*}$ & $Q_5$ & $R_5^{*}$ & $R_5$ \\
\midrule
$16$-cell ($4$-orthoplex)        & 10 &  2 & 10 & 10 &  2 &   2 \\
\treelast octahedral pyramid           & -- & 16 & -- & -- & 16 &  16 \\
rectified $5$-cell               & 32 & 32 & 12 & 12 & 12 &  12 \\
\treemid diminished rectified $5$-cell & -- & -- & 20 & 20 & 20 &  20 \\
\treelast \textit{pyramid over a triangular prism} & -- & -- & 20 & -- & 20 &  -- \\
\midrule
non-unit-edge facets             & -- & -- & -- & 50 & -- &  50 \\
\quad types                      & -- & -- & -- &  2 & -- &   2 \\
\midrule
Total facets                     & 42 & 50 & 62 & 92 & 70 & 100 \\
\midrule
Dents                            & -- & -- & 10 & -- & 10 &  -- \\
\midrule
Unit edges                       & 240 & 240 & 240 & 240 & 240 & 240 \\
Edges at $\langle u,v\rangle = 1/5$ &  -- &  -- &  -- &  10 &  -- &  10 \\
\midrule
Total edges                      & 240 & 240 & 240 & 250 & 240 & 250 \\
\bottomrule
\end{tabular}
\caption{Facets and edges in dimension 5. Internal facet types in italics.}
\label{tab:n5}
\end{table}

\paragraph{\boldmath $Q_5$.} The construction of $Q_5^{*}$ proceeds similarly, but with a different result.
In this case it is the rectified $5$-cell facets of $D_5$ that are split rather than the $16$-cells.
The split is not equatorial,
so the two resulting pieces are distinct: a diminished rectified $5$-cell and a pyramid over a triangular prism.

The resulting pieces can be paired in two new ways
in addition to reconstituting the original rectified $5$-cell.
Two pyramids can join along their bases or two diminished cells can meet along the facet that was produced by the split.
Joining two pyramids forms a dent, two facets at a reflex ridge, and joining two diminished cells produces a convex ridge (Section~\ref{sec:splits}).

The contact-graph polytope $Q_5^{*}$ is obtained by splitting 20 of $D_5$'s 32 rectified $5$-cells
and joining the 20 resulting pyramids to one another in pairs, creating 10
dents. The 20 diminished cells pair up in the same way, creating 10
convex ridges.
As a result, $Q_5^{*}$ is not convex and $Q_5^{*} \neq Q_5$.

The contact polytope $Q_5$ is the convex hull of $Q_5^{*}$.
Its boundary consists of 42 of $Q_5^{*}$'s 62 facets that lie on external hyperplanes, together with 50
new facets that span the dents, each with one non-unit edge, at inner
product $1/5$.

\paragraph{\boldmath $R_5$.} The remaining case results when both types of $D_5$'s facets are split.
The transformations taking $D_5 \to L_5$ and $D_5 \to Q_5^{*}$ are applied
concurrently, and $R_5^{*}$'s census is their combination.
Like $Q_5^{*}$, $R_5^{*}$ is not convex and $R_5^{*} \neq R_5$.
Table~\ref{tab:n5} shows the facet censuses for all four configurations.

\FloatBarrier

\section{Vertex-figure splits}\label{sec:splits}

Pieces obtained by splitting congruent cells can be joined in three ways, as shown
in Figure~\ref{fig:hinge}. A pyramid and a diminished cell recreate the
original cell.
Two diminished cells meet at a convex ridge, and
two pyramids meet at a reflex ridge forming a dent where the surface
folds inward. The dihedral angle at a
convex ridge is $\pi - 2\beta$ and $\pi + 2\beta$ at a reflex
one, and $\beta$ is determined solely by the circumradius of the cell that
was split.

\begin{figure}[H]
\centering
\begin{tikzpicture}[scale=2.05,line join=round,line cap=round]
  \tikzset{
    arm/.style={line width=1.5pt,black!62},
    ref/.style={draw=black!35,line width=0.5pt,dash pattern=on 2.2pt off 1.8pt},
    body/.style={fill=black!8,draw=none},
    ang/.style={draw=black!55,line width=0.5pt},
    lab/.style={font=\scriptsize,align=center,inner sep=1pt},
  }
  \begin{scope}
    \fill[body] (0,0) -- (180:0.26)
       arc[start angle=180,end angle=360,radius=0.26] -- cycle;
    \draw[ang] (180:0.26) arc[start angle=180,end angle=360,radius=0.26];
    \draw[arm] (0,0) -- (180:0.6325);
    \draw[arm] (0,0) -- (0:0.6325);
    \fill (0,0) circle (1.5pt);
    \node[font=\small] at (0,-0.46) {$\pi$};
    \node[font=\footnotesize] at (0,-0.66) {original cell};
  \end{scope}
  \foreach \dx/\la/\ra/\wa/\wb/\bmr/\note/\what in {%
      2.30/194.4775/-14.4775/194.4775/345.5225/-7.24/{$\pi-2\beta$}/{convex ridge},
      4.60/165.5225/14.4775/165.5225/374.4775/7.24/{$\pi+2\beta$}/{reflex ridge}}{
    \begin{scope}[shift={(\dx,0)}]
      \fill[body] (0,0) -- (\wa:0.26)
         arc[start angle=\wa,end angle=\wb,radius=0.26] -- cycle;
      \draw[ang] (\wa:0.26) arc[start angle=\wa,end angle=\wb,radius=0.26];
      \draw[ref] (-0.74,0) -- (0.74,0);
      \draw[arm] (0,0) -- (\la:0.6325);
      \draw[arm] (0,0) -- (\ra:0.6325);
      \node[font=\scriptsize] at (\bmr:0.74) {$\beta$};
      \fill (0,0) circle (1.5pt);
      \node[font=\small] at (0,-0.46) {\note};
      \node[font=\footnotesize] at (0,-0.66) {\what};
    \end{scope}
  }
  \node[lab] at (-0.52,0.34) {diminished\\cell};
  \node[lab] at ( 0.52,0.34) {pyramid};
  \node[lab] at ( 1.78,0.34) {diminished\\cell};
  \node[lab] at ( 2.82,0.34) {diminished\\cell};
  \node[lab] at ( 4.08,0.34) {pyramid};
  \node[lab] at ( 5.12,0.34) {pyramid};
\end{tikzpicture}
\caption{Cross sections of the three ways that pieces of split facets
can be joined}
\label{fig:hinge}
\end{figure}
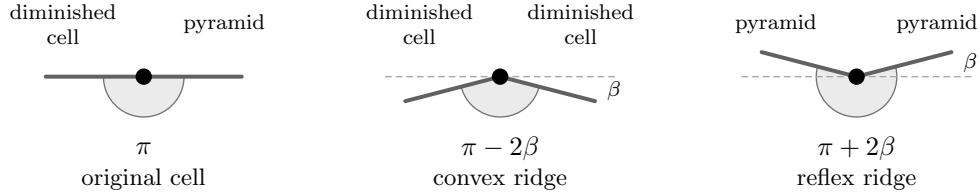

\begin{theorem}[Vertex-figure split]\label{thm:vfsplit}
Let $F$ be a convex unit-edge $(n-1)$-dimensional cell in $\Rn^n$
whose vertices form a kissing configuration, with circumcenter $c_F$ and
circumradius $R_F$. Suppose its affine hull does not contain the origin,
and every pair of vertices at distance $1$ is an edge.
Let $v$ be a vertex of $F$ not
joined to every other vertex, and let $G$ be its vertex figure, with
circumradius $R_G$. Then
\begin{romanlist}
\item the cut is along $\{\, x : \langle x, v \rangle = \tfrac12 \,\}$, and
  $\langle c_F, v \rangle = 1 - R_F^{2}$;
\item $\displaystyle R_G^{2} = 1 - \frac{1}{4R_F^{2}}$;
\item two pieces of splits of cells congruent to $F$, both pyramids or
  both diminished cells, inscribed in $\Sph^{n-1}$ and joined along $G$
  on different hyperplanes, lie on
  hyperplanes that are mirror images in the hyperplane through
  $\operatorname{aff}(G)$ and the origin. The fold angle is $2\beta$, with
  $\sin\beta = 2R_F^{2} - 1$. In the pyramid case, the pieces themselves
  are mirror images;
\item when two pyramids are joined, the apex $v'$ of the second has
  $\langle v, v'\rangle = 2R_F^{2} - 1 = \sin\beta$ and
  $\lvert v - v' \rvert^{2} = 4(1 - R_F^{2})$.
\end{romanlist}
Moreover $R_F^{2} \ge \tfrac12$, and the split is equatorial exactly
when $R_F^{2} = \tfrac12$, and then $\beta = 0$.
\end{theorem}

\begin{figure}[H]
\centering
\begin{tikzpicture}[scale=3.25,line cap=round,line join=round,
  point/.style={circle,fill=black,inner sep=1.2pt},
  every node/.style={font=\small}]
  \pgfmathsetmacro{\foldbeta}{asin(1/3)}
  \coordinate (O) at (0,0);
  \coordinate (cf) at ({1/(3*sqrt(3))},{2*sqrt(6)/9});
  \coordinate (g) at (0,{sqrt(6)/4});
  \coordinate (v) at ({-1/sqrt(3)},{sqrt(2/3)});
  \coordinate (vp) at ({1/sqrt(3)},{sqrt(2/3)});
  \fill[black!6] (O) -- (cf) -- (g) -- cycle;
  \draw[black!45,line width=0.5pt] (1,0)
    arc[start angle=0,end angle=180,radius=1];
  \draw[black!45,dashed,line width=0.5pt] (0,{sqrt(6)/4+0.18}) -- (0,1.08)
    node[above] {$M$};
  \draw[black!50,line width=0.6pt] (v) -- ({5/(3*sqrt(3))},{sqrt(6)/9});
  \draw[black!50,line width=0.6pt] (vp) -- ({-5/(3*sqrt(3))},{sqrt(6)/9});
  \draw[black!85,line width=1.2pt] (v) -- (g) -- (vp);
  \begin{scope}[shift={(g)}]
    \draw[line width=0.5pt] ({180-\foldbeta}:0.24)
      arc[start angle={180-\foldbeta},end angle={180+\foldbeta},radius=0.24];
    \node[left=2pt,yshift=-2pt] at (-0.26,0) {$2\beta$};
  \end{scope}
  \draw[line width=0.6pt] (O) -- (cf)
    node[midway,right=2pt] {$h_F$};
  \draw[line width=0.6pt] (O) -- (g)
    node[midway,left=2pt] {$h_G$};
  \begin{scope}[shift={(cf)},rotate={180-\foldbeta}]
    \draw[line width=0.5pt] (0.055,0) -- (0.055,0.055) -- (0,0.055);
  \end{scope}
  \draw[line width=0.5pt] ({90-\foldbeta}:0.20)
    arc[start angle={90-\foldbeta},end angle=90,radius=0.20];
  \node at ({90-\foldbeta/2}:0.34) {$\beta$};
  \node[rotate=-\foldbeta,above=-2pt] at ({1/sqrt(3)},{sqrt(6)/6})
    {$\operatorname{aff}(F)$};
  \node[point,label={below left:$O$}] at (O) {};
  \node[point,label={[label distance=-2pt,yshift=-1pt]above right:$c_F$}] at (cf) {};
  \node[point,label={above=3pt:$c_G$}] at (g) {};
  \node[point,label={above left:$v$}] at (v) {};
  \node[point,label={above right:$v'$}] at (vp) {};
\end{tikzpicture}
\caption{A section through the sphere's center perpendicular to $G$,
with $\sin\beta=1/3$.
The ridge $G$ meets the section at $c_G$.
The facet hyperplanes are reflected across $M$ with fold angle $2\beta$,
and $O,c_F,c_G$ form a
right triangle with $\cos\beta=h_F/h_G$}
\label{fig:fold-geometry}
\end{figure}

\begin{proof}
Figure~\ref{fig:fold-geometry} shows the geometry in a section perpendicular
to $G$.
The vertices are unit vectors, so $\lvert x - v \rvert = 1$ is
$\langle x, v \rangle = \tfrac12$, and the kissing condition
$\langle x, v \rangle \le \tfrac12$ puts every vertex other than $v$
and the vertices of $G$ strictly below the cut. Since
$\lvert p \rvert^{2} = 1$ for every vertex $p$, the
foot of the perpendicular from the origin to $\operatorname{aff}(F)$ is
equidistant from the vertices, so it is $c_F$.
Thus $h_F^{2} = \lvert c_F \rvert^{2} = 1 - R_F^{2}$, where $h_F$ is the distance
from the origin to $\operatorname{aff}(F)$.
Also, $\langle p, c_F \rangle = h_F^{2}$ for every vertex $p$.
Setting $p = v$ gives the rest of (i).
The same holds for $G$: its circumcenter $c_G$ is the foot
of the perpendicular to $\operatorname{aff}(G)$, at distance $h_G$ with
$h_G^{2} = 1 - R_G^{2}$.

The vertices of $F$ lie on a sphere of radius $R_F$ about $c_F$ in
$\operatorname{aff}(F)$, at distance at least $1$ from one another, so
$\langle p - c_F, q - c_F \rangle \le R_F^{2} - \tfrac12$ for distinct
vertices $p$ and $q$. If $R_F^{2} < \tfrac12$ these inner products are
all negative, and at most $n$ vectors in the $(n-1)$-dimensional space
$\operatorname{aff}(F) - c_F$ have pairwise negative inner products, so $F$
would be a simplex, contrary to the choice of $v$.
Hence $R_F^{2} \ge \tfrac12$.

Within $\operatorname{aff}(F)$, the cutting hyperplane has normal $v-c_F$,
whose length is $R_F$, so the cut lies at distance
$d = \lvert \tfrac12 - (1 - R_F^{2}) \rvert / R_F$ from $c_F$.
Therefore $R_G^{2} = R_F^{2} - d^{2} = 1 - 1/(4R_F^{2})$, which is (ii).
It follows that $h_G^{2} = 1/(4R_F^{2})$.

Let $M$ be the hyperplane through $\operatorname{aff}(G)$ and the origin.
Reflection in $M$ fixes $G$ pointwise and preserves the sphere, so it
carries either piece to a congruent piece inscribed in the same
sphere and meeting the original piece along $G$. Conversely, a piece of a split of a cell
congruent to $F$, cut at a vertex $u$ with vertex figure $G$, lies on the
hyperplane through $G$ and $u$, and $u$ is a unit vector at distance $1$
from every vertex of $G$. Since $\operatorname{aff}(G)$ spans $M$, these
conditions fix the projection of $u$ onto $M$, so $u$ is $v$ or its
mirror image $v'$ in $M$. A second piece on a different hyperplane from
the first therefore lies on the mirror image of $\operatorname{aff}(F)$,
and a second pyramid, being $G$ and its apex, is the mirror image of the
first.

The triangle with vertices at the origin, $c_F$, and $c_G$
is right-angled at $c_F$, so the normals $c_F$ and $c_G$ of the two affine
hulls meet at the angle $\beta$ with $\cos\beta = h_F/h_G$, and the
reflection turns $\operatorname{aff}(F)$ about $\operatorname{aff}(G)$
through $2\beta$. So
$\sin^{2}\beta = 1 - h_F^{2}/h_G^{2} = 1 - 4R_F^{2}(1 - R_F^{2})
 = (2R_F^{2} - 1)^{2}$, and $\sin\beta = 2R_F^{2} - 1$ since
$R_F^{2} \ge \tfrac12$, which is (iii).

For (iv), $v'$ is the reflection of $v$ in $M$, so
$\langle v, v' \rangle = 1 - 2\,\mathrm{dist}(v, M)^{2}$. The projection
of $v$ onto $M$ is $c_G/(2h_G^{2})$. That point lies in $M$ and has inner
product $\tfrac12$ with every point of $\operatorname{aff}(G)$, as $v$ does,
since $\langle c_G, x \rangle = h_G^{2}$ there.
Since $\operatorname{aff}(G)$ spans $M$, those inner products determine the
projection. Its squared length is $1/(4h_G^{2}) = R_F^{2}$, so
$\mathrm{dist}(v, M)^{2} = 1 - R_F^{2}$,
$\langle v, v' \rangle = 2R_F^{2} - 1$, and
$\lvert v - v' \rvert^{2} = 2 - 2\langle v, v' \rangle = 4(1 - R_F^{2})$.
Finally $d = 0$ exactly when $R_F^{2} = \tfrac12$, and then $h_G = h_F$
and $\beta = 0$.
\end{proof}

\paragraph{The lattice facets.} Table~\ref{tab:cuts} lists the facets of the lattice contact polytopes
from dimension 4 through 8 and the splits they admit. The
simplices cannot be split because the pyramid would be the whole cell.
The orthoplexes have $R^{2} = \tfrac12$,
so they only split equatorially and their pieces can only join flat.
The remaining three facet types can fold: the rectified $5$-cell, the $5$-demicube,
and the Gosset $2_{21}$.

\begin{table}[!ht]
\centering\footnotesize
\setlength{\tabcolsep}{3.5pt}
\renewcommand{\arraystretch}{1.25}
\begin{tabular}{llllccc}
\toprule
Config.\ & Facet $F$ & Vertex figure $G$ & Split pieces
    & $R_F^{2}$ & $\sin\beta$ & $\ell$ \\
\midrule
$D_4$ & octahedron & square & equatorial & $1/2$ & $0$ & --- \\
\midrule
\multirow{2}{*}{$D_5$}
  & $16$-cell & octahedron & equatorial & $1/2$ & $0$ & --- \\
  & rectified $5$-cell ($0_{21}$) & triangular prism ($-1_{21}$)
    & diminished $+$ pyramid & $3/5$ & $1/5$ & $\sqrt{8/5}$ \\
\midrule
$E_6$ & $5$-demicube ($1_{21}$) & rectified $5$-cell ($0_{21}$)
  & diminished $+$ pyramid & $5/8$ & $1/4$ & $\sqrt{3/2}$ \\
\midrule
\multirow{2}{*}{$E_7$}
  & $6$-simplex & $5$-simplex & none & & & \\
  & Gosset $2_{21}$ & $5$-demicube ($1_{21}$) & diminished $+$ pyramid
    & $2/3$ & $1/3$ & $\sqrt{4/3}$ \\
\midrule
\multirow{2}{*}{$E_8$}
  & $7$-simplex & $6$-simplex & none & & & \\
  & $7$-orthoplex & $6$-orthoplex & equatorial & $1/2$ & $0$ & --- \\
\bottomrule
\end{tabular}
\caption{The facets of the lattice contact polytopes and the splits they
admit. The pieces of a folding split fold through $2\beta$, and
$\ell$ is the length of the non-unit edge between the apexes of the two
pyramids forming the dent.}
\label{tab:cuts}
\end{table}

\paragraph{\boldmath The $k_{21}$ series.} The three folding facets are consecutive members of the series $k_{21}$,
each vertex figure being the previous member $(k-1)_{21}$, and for this
series the fold angle has a closed form. The circumradii of the Gosset
polytopes are~\cite{Coxeter1988}
\[
  R^{2}(k_{21}) \;=\; \frac{6-k}{2(5-k)},
  \qquad n = k + 5,
\]
so by Theorem~\ref{thm:vfsplit}
\[
  \sin\beta \;=\; 2R^{2}(k_{21}) - 1 \;=\; \frac{1}{5-k} \;=\; \frac{1}{10-n}.
\]

\begin{corollary}\label{cor:angle}
A vertex-figure split of the $k_{21}$ facet of the lattice contact polytope
in dimension $n = k+5$, for $k = 0$, $1$, and $2$, folds through $2\beta$
with $\sin\beta = 1/(10-n)$,
and the two apexes of the dent it creates have inner product
$1/(10-n)$.
\end{corollary}

In dimensions 5, 6, and 7, the Gosset facet splits fold through $2\beta$ with
$\sin\beta = 1/5$, $1/4$, and $1/3$, respectively. The inner product between
the two apexes is the same number. The lattice configurations have inner
products $0$, $\pm\tfrac12$, and $\pm1$ only, so each dent creates a pair
at a new inner product, $1/(10-n)$. These are the inner products in the
last column of Table~\ref{tab:all}.

Among the lattice facets from dimension 4 to 8, a split can fold only in
dimensions 5, 6, and 7. No folding splits are possible at $n=4$ or $8$.

\FloatBarrier

\section{Dimension 6}\label{sec:n6}

Dimension 6 continues the pattern of
dimension 5: facets of the lattice contact polytope are split, and the pieces are
joined in pairs, forming dents and convex ridges.
The contact polytope of $E_6$ is the Gosset $1_{22}$, whose 54 facets are
all $5$-demicubes.
The $5$-demicube splits into a diminished $5$-demicube and a pyramid over a
rectified $5$-cell.

The difference between $C6\text{-}72b^{*}$, $C6\text{-}72d^{*}$, and $C6\text{-}72c^{*}$ is how
many of the 54 facets are split and how many times. A facet may be split twice
by removing a second vertex from the diminished cell, producing a second
pyramid and a twice-diminished cell.
For $C6\text{-}72b^{*}$, 32 facets are split once, creating 16 dents.
For $C6\text{-}72d^{*}$, 32 are split once and 16 twice, creating 32 dents.
For $C6\text{-}72c^{*}$, 48 are split twice, creating 48 dents.

All three contact-graph polytopes are non-convex, and the non-unit edges of
$P$ are all at inner product $1/4$.
Table~\ref{tab:n6} shows the facet censuses.

\begin{table}[ht]
\centering
\setlength{\tabcolsep}{3.2pt}
\begin{tabular}{lrrrrrrr}
\toprule
Facets & $E_6$ & $C6\text{-}72b^{*}$ & $C6\text{-}72b$ & $C6\text{-}72d^{*}$ & $C6\text{-}72d$
 & $C6\text{-}72c^{*}$ & $C6\text{-}72c$ \\
\midrule
$5$-demicube                       & 54 & 22 &  22 &  6 &   6 &  6 &   6 \\
\treemid diminished $5$-demicube         & -- & 32 &  32 & 32 &  32 & -- &  -- \\
\treepass\treemid $2$-diminished $5$-demicube & -- & -- &  -- & 16 &  16 & 48 &  48 \\
\treelastto\treejoin \textit{pyramid over rectified $5$-cell} & -- & 32 &  -- & 64 &  -- & 96 &  -- \\
\midrule
non-unit-edge facets                & -- & -- & 160 & -- & 304 & -- & 448 \\
\quad types                         & -- & -- &   2 & -- &   3 & -- &   3 \\
\midrule
Total facets                        & 54 & 86 & 214 &118 & 358 &150 & 502 \\
\midrule
Dents                               & -- & 16 &  -- & 32 &  -- & 48 &  -- \\
\midrule
Unit edges                          &720 &720 & 720 &720 & 720 &720 & 720 \\
Edges at $\langle u,v\rangle = 1/4$ & -- & -- &  16 & -- &  32 & -- &  48 \\
\midrule
Total edges                         &720 &720 & 736 &720 & 752 &720 & 768 \\
\bottomrule
\end{tabular}
\caption{Facets and edges in dimension 6. Internal facet types in italics.}
\label{tab:n6}
\end{table}

\FloatBarrier

\section{Dimension 7}\label{sec:n7}

The contact polytope of $E_7$ is the Gosset $2_{31}$, with
576 $6$-simplices and 56 Gosset $2_{21}$ facets. Only the $2_{21}$ facets can be split, not the simplices.
Table~\ref{tab:n7} collects the facet censuses.

\paragraph{\boldmath $C7\text{-}126c$.} $C7\text{-}126c$ continues the pattern of dimensions 5 and 6: splitting 54 of the 56 $2_{21}$ facets and reassembling
the pieces gives $C7\text{-}126c^{*}$ with 27 dents. Its convex hull,
$C7\text{-}126c$, has 27 additional edges at inner product $1/3$.

\paragraph{\boldmath $C7\text{-}126b$.} $C7\text{-}126b$ is quite different
from the previous cases. Its contact-graph polytope $C7\text{-}126b^{*}$
cannot result from splitting and
reassembling the facets of $E_7$ because it has fewer simplex facets than $E_7$ does,
and simplices cannot be split. The convex hull $C7\text{-}126b$ also has non-unit edges
at inner product $1/4$, which is the inner product produced by facet splits in dimension 6,
not 7.

The explanation is that
$E_7$ has 28 equatorial sections parallel to its Gosset facets that are
copies of $E_6$.
The configuration $C7\text{-}126b$ is obtained by
simultaneously replacing 12 of the $E_6$ sections by copies of
$C6\text{-}72b$ in the same hyperplanes. Each replacement retains 56 of the
72 vertices and replaces the other 16, producing 16 pairs at inner product
$1/4$. These account for all 192 such pairs, each occurring
in exactly one of the 12 sections.
Thus the inner product $1/4$ comes indirectly from the split-and-reassemble
construction in dimension 6, through the $C6\text{-}72b$ sections.

The resulting contact-graph polytope $C7\text{-}126b^{*}$ has two
unit-edge facet types absent from $E_7$: a pyramid over a rectified
$5$-simplex, and a $2_{21}$ with six vertices removed. These internal
facets meet in pairs at reflex ridges, forming 64 dents.

\paragraph{\boldmath $C7\text{-}126d$.} $C7\text{-}126d$ returns to the
original pattern: its contact-graph polytope $C7\text{-}126d^{*}$ is
obtained by splitting and reassembling the facets of $C7\text{-}126b^{*}$
rather than $E_7$. In all, 54
facets are split: 22 of the 24 Gosset facets and 32
of the 64 $6$-diminished $2_{21}$ facets.
The convex hull $C7\text{-}126d$ retains edges at inner product $1/4$
from $C7\text{-}126b$ and has additional edges at inner product $1/3$
spanning the new dents.

Because the $6$-diminished $2_{21}$ facets are internal facets of the dents of
$C7\text{-}126b^{*}$, this results in dents inside dents.

\begin{table}[!ht]
\centering\small
\setlength{\tabcolsep}{2.3pt}
\begin{tabular}{lrrrrrrr}
\toprule
Facets & $E_7$ & $C7\text{-}126c^{*}$ & $C7\text{-}126c$
 & $C7\text{-}126b^{*}$ & $C7\text{-}126b$ & $C7\text{-}126d^{*}$ & $C7\text{-}126d$ \\
\midrule
$6$-simplex                        & 576 & 576 &  576 &  384 &  384 &  384 &  384 \\
Gosset $2_{21}$                    &  56 &   2 &    2 &   24 &   24 &    2 &    2 \\
\treemid diminished $2_{21}$             &  -- &  54 &   54 &   -- &   -- &   22 &   22 \\
\treelast \textit{pyramid over $5$-demicube} &  -- &  54 &   -- &   -- &   -- &   22 &   -- \\
\treetop \textit{pyramid over rect.\ $5$-simplex} &  -- &  -- &   -- &   64 &   -- &   64 &   -- \\
\treelast \textit{$6$-diminished $2_{21}$}  &  -- &  -- &   -- &   64 &   -- &   32 &   -- \\
\treeblank\treemid \textit{$7$-diminished $2_{21}$} &  -- &  -- &   -- &   -- &   -- &   32 &   -- \\
\treeblank\treelast \textit{pyramid over dim.\ $5$-demicube} &  -- &  -- &   -- &   -- &   -- &   32 &   -- \\
\midrule
non-unit-edge facets               &  -- &  -- &  702 &   -- & 1,968 &   -- & 2,414 \\
\quad types                        &  -- &  -- &    2 &   -- &    4 &   -- &   11 \\
\midrule
Total facets                       & 632 & 686 & 1,334 &  536 & 2,376 &  590 & 2,822 \\
\midrule
Dents                              &  -- &  27 &    -- &   64 &    -- &   91 &    -- \\
\midrule
Unit edges                         &2,016 &2,016 & 2,016 & 1,984 & 1,984 & 1,984 & 1,984 \\
Edges at $\langle u,v\rangle = 1/3$ &  -- &  -- &   27 &   -- &   -- &   -- &   27 \\
Edges at $\langle u,v\rangle = 1/4$ &  -- &  -- &   -- &   -- &  192 &   -- &  176 \\
\midrule
Total edges                        &2,016 &2,016 & 2,043 & 1,984 & 2,176 & 1,984 & 2,187 \\
\bottomrule
\end{tabular}
\caption{Facets and edges in dimension 7. Internal facet types in italics.}
\label{tab:n7}
\end{table}

\FloatBarrier

\section{Rigidity}\label{sec:rigid}

The contacts of each configuration, which are the edges of $P^{*}$
(Theorem~\ref{thm:exist}), hold it rigid in two ways. For rigidity of the
skeleton, the vertices may leave the sphere, but the contact lengths remain
fixed. For jamming, the vertices remain on the sphere, and contact distances
may increase but not decrease. We verify infinitesimal
rigidity of all 12 skeletons and extend the jamming results of Cohn,
Jiao, Kumar, and Torquato~\cite{CJKT2011} to $Q_5$ and $R_5$.

\paragraph{Rigidity of the skeleton.} A
bar framework is \emph{infinitesimally rigid} when every first-order
motion of its vertices that preserves the lengths of the bars is an
infinitesimal Euclidean motion.

\begin{proposition}\label{prop:skel}
The contact graph of each of the 12 configurations, the $1$-skeleton of
$P^{*}$, is infinitesimally rigid in $\Rn^{n}$ as a bar
framework with the vertices free.
\end{proposition}

\begin{proof}
The rigidity matrix has a row for each edge $ij$, carrying $p_i - p_j$ in
the columns of $i$ and $p_j - p_i$ in those of $j$. Its rank is at most
$n|V| - n(n+1)/2$, since the infinitesimal Euclidean motions are in its
kernel, and the framework is infinitesimally rigid exactly when the rank
is that large.

Using the exact coordinates described in Section~\ref{sec:method}, let $M$
be the integer matrix formed from the coordinate differences $x_i-x_j$.
With positive diagonal form $D$ and
common squared norm $N$, the Euclidean coordinates are
$p_i = D^{1/2}x_i/\sqrt N$.
An invertible scaling of the columns of $M$ therefore gives the Euclidean
rigidity matrix, so their ranks agree. Coordinate directions outside the
$n$-dimensional span contribute zero columns after a change of basis.
The rank of $M^{\mathsf T}M$ modulo a prime is a lower bound on its
rational rank, which equals that of $M$.
For this computation we use the prime
$p = 2^{31}-1$.
The ranks are 185 in dimension 5, 411 in dimension 6, and 854 in
dimension 7, the required bounds for all 12 configurations.
\end{proof}

The contacts alone make each configuration infinitesimally rigid, even
when the vertices may leave the sphere and the facets may deform.
This also holds for the non-convex polytopes, where rigidity theorems
for convex polytopes do not apply.

This contrasts with dimension 3. The cuboctahedron, the contact polytope
of the face-centered cubic configuration, is rigid with its faces held
rigid by Cauchy's theorem. Its $1$-skeleton, however, has only 24 edges,
whereas infinitesimal rigidity requires at least 30. The configuration
is also not jammed~\cite[Proposition~3.5]{CJKT2011}.

\paragraph{Jamming.} Cohn, Jiao, Kumar, and Torquato~\cite{CJKT2011} call
a spherical code \emph{infinitesimally jammed} when every first-order
motion of its points that keeps them on the sphere and lets no contact
distance decrease is an infinitesimal rotation. An infinitesimally jammed
code is jammed~\cite[Theorem~2.2]{CJKT2011}. They verified that the 10
configurations they list are infinitesimally jammed. $Q_5$ and $R_5$ were
found later, and Sz\"{o}ll\H{o}si asks about $Q_5$~\cite{Szollosi2023}.

\begin{theorem}[Jamming]\label{thm:jam}
$Q_5$ and $R_5$ are infinitesimally jammed.
\end{theorem}

\begin{proof}
Hold each point to the sphere by a bar to the center, and make each
contact a strut, a member that may lengthen but not shorten. The
configuration is infinitesimally jammed exactly when this tensegrity is
infinitesimally rigid. By Roth and
Whiteley~\cite[Theorem~5.2]{RothWhiteley1981}, this is equivalent to two
conditions: the framework with every member a bar is infinitesimally
rigid, and there is an equilibrium stress that is negative on every
strut. Such a stress consists of numbers $w_{ij}$ on the contacts and
$s_i$ on the bars with
\[
\sum_{j} w_{ij}\,(p_i - p_j) + s_i\, p_i = 0 \text{ at every point } p_i,
\qquad \sum_i s_i\, p_i = 0 \text{ at the center,}
\]
and every $w_{ij} < 0$. The bar-framework condition follows from
Proposition~\ref{prop:skel}. The framework is the contact graph with one
more vertex, the center, joined to every point, and the points span
$\Rn^{n}$, so the new bars add $n$ to the rank. For the stress condition, a stress
with $w_{ij} < 0$ on all 240 contacts, found by linear programming, is
given for each of $Q_5$ and $R_5$ and checked in integer arithmetic.
The certificates are \texttt{rigidity/Q5.stress.txt} and
\texttt{rigidity/R5.stress.txt} in the verification archive.
\end{proof}

Together with the 10 cases proved by Cohn, Jiao, Kumar, and
Torquato~\cite{CJKT2011}, Theorem~\ref{thm:jam} establishes infinitesimal
jamming for all 12 configurations.

\FloatBarrier

\section{Proof of Theorems~\ref{thm:exist}--\ref{thm:parts}}\label{sec:method}

\paragraph{Coordinates.} Following Cohn, Jiao, Kumar, and
Torquato~\cite{CJKT2011}, we use integer coordinates $x$, a positive
diagonal form $D$, and common squared norm $N=x^{\mathsf T}Dx$.
Normalized inner products are $x^{\mathsf T}Dy/N$. All mathematical
computations use integer or rational arithmetic. Each configuration is matched
to the published coordinates up to relabeling: \cite{CJKT2011} for 10,
and \cite{CohnR2024} for $Q_5$ and $R_5$.

\paragraph{Existence.} Each polytope is exhibited as a finite list of
vertex subsets of $V$. Each subset defines a facet as its convex hull.
We check four conditions:
\begin{enumerate}
\item every facet is unit-edge, has dimension $n-1$, and has a
  hyperplane that misses the origin;
\item the facet vertex sets cover exactly $V$;
\item every ridge is a face of exactly two facets, on different
  hyperplanes;
\item the union of the facets' edge sets is exactly the set of contacts
  of $V$.
\end{enumerate}
Finally, we check that the cones from the origin over distinct facets
have disjoint interiors. Cone interiors intersect exactly when a point
is a positive combination of both vertex sets, which is decided by
rational linear feasibility. Together with the facet and ridge conditions,
Proposition~\ref{prop:star} establishes that the facets form an embedded,
face-to-face boundary. It also shows that the resulting polytope is
radial about the sphere's center. The enclosed region has vertex set
$V$, unit edges, and no flat ridges, as required. The fourth condition
establishes that it is a contact-graph polytope.

Every contact is an edge of $\conv V$. For a contact pair $u,v$,
the functional $x\mapsto\langle u+v,x\rangle$ equals $3/2$ at $u$
and $v$ and is at most $1$ at every other point of $V$. It therefore
exposes the edge $[u,v]$. Thus a unit-edge convex hull is a
contact-graph polytope.

Comparison with the hull enumeration below confirms that the external
facets of $P^{*}$ are exactly the unit-edge facets of $\conv V$.
For $D_5$, $L_5$,
$E_6$, and $E_7$, the complete facet lists of $P^{*}$ and $P$ agree.
For the other eight, $P$ has longer edges. Since $\conv V$ is the
only convex polytope on $V$, their $P^{*}$ are non-convex. This proves
Theorem~\ref{thm:exist}.

\paragraph{Radial polytopes.} The exhibited polytopes are radial about
the sphere's center. Their projected facets form a face-to-face
decomposition of the sphere, with contact edges projected to shorter
great-circle arcs. Radiality retains this property of convex hulls
containing the center in their interior while allowing non-convex
boundaries.

\paragraph{Enumeration.} The enumeration recovers the exhibited polytopes
and proves that each is the unique radial contact-graph polytope on its
vertex set. It is finite: every facet lies in a hyperplane spanned by
configuration points. Symmetry reduction and pruning make it tractable.
It proceeds in four steps:
\begin{enumerate}
\item Enumerate candidate hyperplanes that could contain unit-edge facets,
  up to symmetries of $V$.
\item Enumerate candidate facets within each hyperplane $H$: subsets of $V\cap H$ whose
  convex hulls could be facets of a radial unit-edge polytope.
\item Prune by ridge partners: repeatedly remove candidates lacking a
  compatible partner at some ridge.
\item Classify the surviving boundaries: determine which can be assembled from the surviving
  facets and compare their edges with the contacts.
\end{enumerate}

\subsection*{Step 1: Enumerate candidate hyperplanes}

At each vertex $v$ of a unit-edge facet,
the incident edges join $v$ to contact neighbors and span the facet's
directions. These neighbors lie in a common hyperplane,
$\langle x,v\rangle=\tfrac12$, so they span an $(n-2)$-dimensional
affine subspace. Together with $v$, this subspace determines the facet's
hyperplane. We enumerate all such subspaces spanned by contact neighbors
of $v$, so this procedure finds every hyperplane through $v$ that could
contain a unit-edge facet.
To reduce duplication, we use one vertex from each orbit under
$G=\Aut(V)$, then retain one hyperplane from each orbit under $G$.

\subsection*{Step 2: Enumerate candidate facets}

Within each candidate hyperplane $H$, we enumerate every subset of
$V\cap H$ that satisfies the facet requirements, up to symmetry.
Necessary conditions eliminate impossible choices before and during
the subset search.

\paragraph{Facet requirements.} For a hyperplane $H$, the
section of $V$ in $H$ is the point set $V\cap H$.
A candidate facet is a subset $C\subseteq V\cap H$ spanning $H$ whose convex
hull is unit-edge. Hyperplanes through the origin are excluded, since
radial projection cannot be injective on such a facet.
The candidate's closed convex hull must also avoid the central
projection onto $H$ of every point of $V$ on the side of $H$ opposite
the origin, including projections on the candidate's boundary.
Otherwise a ray through such a point
would meet the boundary twice.

\paragraph{Symmetry reduction.} Each hyperplane $H$ is searched
separately, using its stabilizer $G_H$ in $G=\operatorname{Aut}(V)$.
These symmetries preserve contacts and the conditions imposed by radiality, so they
preserve the candidate facets.

Fix an order on the vertices of $V\cap H$, and write each subset in
increasing order. The selected vertices form a prefix of any completion
written in this order. The search keeps only lexicographically least subsets in their
$G_H$-orbits. Every increasing prefix of such a subset is itself least:
a smaller image of the prefix would give a smaller image of the whole
subset. Branching chooses each
possible next vertex in increasing order, excluding all smaller unselected
vertices. Each prefix is also tested as a completed candidate. These choices
partition the completions, and the least representative of every valid
orbit is reached through its successive prefixes.

Two candidates in $H$ are equivalent under $G$ exactly when they are
equivalent under $G_H$, since a symmetry between them must preserve their
span $H$. Thus one representative per $G_H$-orbit in each hyperplane
representative gives one per facet orbit under $G$.
For a candidate $C$ with stabilizer $G_C$ in $G$, the code checks its
orbit size in $H$ against $|G_H|/|G_C|$.
Its orbit under $G$ has size $|G|/|G_C|$.

\paragraph{Contact-cone criterion.} The following criterion prunes the
remaining prefixes. For points on a sphere, the convex
hull is unit-edge exactly when the directions to unit-distance neighbors
at each vertex generate its tangent cone.
\begin{lemma}\label{lem:contact-cone}
Let $C$ be a finite set of distinct points on a sphere. The convex hull
of $C$ is unit-edge if and only if, for every $v,w\in C$,
\[
 w-v\in\operatorname{cone}\{u-v:u\in C,\ |u-v|=1\}
 .
\]
\end{lemma}
\begin{proof}
Necessity follows because the edges at $v$ generate its tangent cone.
Conversely, these inclusions say that the selected contact directions
generate the whole tangent cone. Every extreme ray then contains a
contact neighbor. Distinct points on a sphere cannot lie on the same
ray from another point of that sphere, so every edge at $v$ has length
one.
\end{proof}

\paragraph{Selected and allowed vertices.} At each step, the prefix
forms the selected set $T$, and the allowed vertices form a set $W$
containing $T$. Every completion $C$ must
satisfy $T\subseteq C\subseteq W$.
Initially, a vertex is excluded if its available contact directions
span fewer than $n-1$ dimensions, since they must span the facet's hyperplane.

Lemma~\ref{lem:contact-cone} restricts these choices. For each ordered
pair of distinct selected vertices $v,w$, the direction $w-v$ must be a
nonnegative combination of directions to contact neighbors of $v$ in the
completed facet. The code tests whether
\[
 w-v\in\operatorname{cone}\{u-v:u\in W,\ |u-v|=1\}
\]
by exact linear feasibility. If this fails, no completion within $W$ is
possible.
The code also rejects a branch when the selected vertices violate the
facet requirements imposed by radiality. These conditions are necessary for a completion, so they
discard no candidate.

\subsection*{Step 3: Prune by ridge partners}

A candidate is deleted if one of
its ridges has no compatible partner, and deletion is repeated to a
fixed point. A partner has the same ridge as an actual face, lies on a
different hyperplane, and has nonoverlapping radial wedges at
every peak of the ridge. In spherical projection, a small two-dimensional
slice transverse to a peak cuts each incident facet into a wedge bounded
by its two incident ridges. These wedges must have disjoint interiors
in a radial boundary. The facet across a ridge in any valid boundary remains a
partner, so pruning deletes no facet of a polytope.

The candidate facets are grouped into orbits under $G$.
For each representative $F$, it suffices to test one ridge from each
orbit under its stabilizer $G_F$, since these symmetries preserve partner
compatibility. The resulting incident facet--ridge pairs are called flags.
To compare ridges from different facet orbits, choose a representative
$R$ of each ridge orbit under $G$. Each flag whose ridge belongs to
this orbit is mapped by a symmetry in $G$ so that its ridge becomes $R$.
The stabilizer $G_R$ of $R$ then gives all distinct placements of the
facets along $R$. Distinct facets in the same orbit under $G$ may be
compatible partners.

Deletion is invariant under $G$, so removing unsupported facet orbits
gives the same fixed point as pruning all individual candidate facets.

Table~\ref{tab:searches} counts candidate facets, their orbits under
$\Aut(V)$, and individual facets surviving ridge-partner pruning.

\begin{table}[!ht]
\centering\small
\setlength{\tabcolsep}{6pt}
\renewcommand{\arraystretch}{1.1}
\begin{tabular}{clrrr}
\toprule
$n$ & Config. & Candidate facets & Candidate orbits & Surviving facets \\
\midrule
\multirow{4}{*}{$5$} & $D_5$ & $1{,}482$ & $7$ & $42$ \\
 & $L_5$ & $1{,}330$ & $19$ & $50$ \\
 & $Q_5$ & $922$ & $21$ & $62$ \\
 & $R_5$ & $770$ & $49$ & $70$ \\
\midrule
\multirow{4}{*}{$6$} & $E_6$ & $122{,}670$ & $18$ & $54$ \\
 & $C6\text{-}72b$ & $72{,}494$ & $106$ & $86$ \\
 & $C6\text{-}72d$ & $40{,}318$ & $365$ & $118$ \\
 & $C6\text{-}72c$ & $25{,}950$ & $55$ & $150$ \\
\midrule
\multirow{4}{*}{$7$} & $E_7$ & $70{,}699{,}400$ & $194$ & $26{,}840$ \\
 & $C7\text{-}126c$ & $24{,}204{,}392$ & $935$ & $686$ \\
 & $C7\text{-}126b$ & $30{,}845{,}928$ & $1{,}884$ & $536$ \\
 & $C7\text{-}126d$ & $11{,}759{,}992$ & $7{,}132$ & $590$ \\
\bottomrule
\end{tabular}
\caption{Candidate facets and survivors of ridge-partner pruning.}
\label{tab:searches}
\end{table}

\FloatBarrier

\subsection*{Step 4: Classify the surviving boundaries}

For 11 configurations, the survivors are exactly the facets of $P^{*}$.
Any radial unit-edge polytope on $V$ must use only these facets.
Each ridge has exactly two incident facets, so including one facet forces
its neighbors and then the whole connected boundary. Thus $P^{*}$ is
the only radial unit-edge polytope on each of these vertex sets.

For $E_7$, extra facets survive: facets not belonging to its convex
contact polytope. Complete ridge-partner checks show that any boundary
using an extra facet must contain a facet from one fixed orbit under $G$.
Starting with a representative of that orbit, repeatedly adding the only
compatible partner at a ridge forces
the verified boundary of a second polytope, which we call the
\emph{augmented $A_7$ polytope}, $A_7^{+}$. It is obtained by attaching
a unit-edge pyramid to each of the 70 facets of the $A_7$ root polytope
that are products of two tetrahedra.
Every ridge must belong to exactly two facets, and facet adjacency must
be connected, so no larger boundary can contain it.
The polytope $A_7^{+}$ is non-convex and unit-edge, but it
omits 560 contacts of $E_7$, so it is not a contact-graph polytope.
Without extra facets, ridge adjacency forces the convex boundary, so
$E_7$ is the unique radial contact-graph polytope on its vertex set.
Together with the other 11 cases, this proves Theorem~\ref{thm:unique}.
The supplied \texttt{check\_e7\_forcing.py} checks the lists of compatible
ridge partners in
\texttt{E7-partner-counts.json} and the additions in
\texttt{E7-forced-closure.json} against all surviving facets.

\subsection*{Convex hulls, facets, and reassembly}

\paragraph{The convex hulls.} Hull faces are enumerated by extending
affine subspaces only when some containing hyperplane has all points of $V$
outside the subspace strictly on one side. Rational linear feasibility
decides each extension. The maximal proper faces give the hull facets.

\paragraph{Naming the facets.} Exact coordinates are supplied for each
facet type named in Tables~\ref{tab:n5},~\ref{tab:n6}, and~\ref{tab:n7}.
Facets are identified by matching their Gram matrices to these coordinate
models, up to a relabeling of the vertices. In these names
\emph{$k$-diminished} means $k$ vertices removed, in Johnson's
sense~\cite{Johnson1966}, and a pyramid is named by its base.

\paragraph{Splits and reassembly.}
\begin{proof}[Proof of Theorem~\ref{thm:parts}]
Each vertex-figure split is checked using exact coordinates for the cell
being split. The neighbors
of the chosen vertex span the cutting hyperplane, all other vertices
except the apex lie strictly on its other side, and the two pieces
match the named cells by exact Gram matrices and circumradii.
Since every edge at the apex ends on the cut, the cut introduces no
new vertices. The two pieces therefore divide the parent cell along
the vertex figure.

The numbers of splits in Sections~\ref{sec:n5}, \ref{sec:n6},
and~\ref{sec:n7} are then applied to the complete facet multiset of the
source polytope. The resulting multiset equals the complete facet
multiset of the target, with the multiplicities in
Tables~\ref{tab:n5}, \ref{tab:n6}, and~\ref{tab:n7}.
Each detached piece can therefore be moved by a rigid motion to its
matching target facet, using every piece once. The target coordinates
and facet lists give the placement. The existence checks have already
verified the target boundary and equality of its edges with the contacts.
\end{proof}

For the configurations obtained by splits, the verifier also compares
all pairs at inner product $1/(10-n)$ with the apex pairs of pyramids
sharing a base. The two sets agree, so the fold angles account for
every pair at this inner product.

\paragraph{Equatorial sections.} The program \texttt{equators.py} verifies
the transformation in Section~\ref{sec:n7}: reflecting a 32-vertex layer
of $E_7$, with the other 94 vertices fixed, gives $C7\text{-}126b$.
Exact Gram-matrix comparisons identify $E_7$ and the $E_6$ and
$C6\text{-}72b$ sections in 12 specified hyperplanes. Set comparisons
verify that the section replacements account for the entire vertex
change and partition all 192 pairs at inner product $1/4$, with 16
in each section.

\paragraph{Reproduction.} From the supplied \texttt{verify/} directory, run
\begin{verbatim}
python3 verify_paper.py all
\end{verbatim}
This recomputes all candidate searches and checks the boundaries,
contacts, censuses, splits, rigidity, and $E_7$ classification. Every
requested check must pass. The programs use only the Python standard
library. The source, data, facet lists, certificates, and README accompany
the paper. The README documents implementation details, checks against
direct enumeration, and run records tied to exact source and data hashes.

The command above uses one worker. With four worker processes
(\texttt{-j4}), the complete verification of all 12 configurations took
about eight minutes on a MacBook Air, including enumeration, pruning, boundary
classification, and the other geometric checks. The supplied run records
identify the source and data used.

\FloatBarrier

\section{Conclusion}\label{sec:discussion}

Given a set of points and prescribed edges, we can ask whether they are
the vertices and edges of a polytope, and whether such a polytope is unique.
For each of the 12 known conjecturally optimal kissing configurations in
dimensions 5, 6, and 7, a contact-graph polytope exists. These polytopes
are radial, and each is unique within that class. Eight are non-convex,
as required when the convex hull includes edges longer than the contacts.

The contact graphs of all 12 configurations are infinitesimally rigid,
even when their vertices may leave the sphere. The jamming certificates
for $Q_5$ and $R_5$ extend the 10 cases established by Cohn, Jiao, Kumar,
and Torquato~\cite{CJKT2011} to all 12 configurations.

The contact-graph polytopes describe these configurations through their
facets, complementing the accounts by layers, cross sections, and
colorings. Seven of the nine non-lattice cases are obtained by splitting
and reassembling facets of the lattice contact polytope in the same
dimension. These constructions connect the non-lattice configurations to
the lattice ones. The fold angles of the resulting dents explain the inner products
$1/5$, $1/4$, and $1/3$ in dimensions 5, 6, and 7, respectively.

\FloatBarrier

\appendix
\section{Radial projection}\label{app:star}

This appendix gives a sufficient condition for a radial polytope boundary. It applies to a candidate family of facets
before the topology of its union has been established.

\begin{proposition}\label{prop:star}
Let $n \ge 3$, and let $\mathcal F$ be a nonempty finite family of convex
$(n-1)$-dimensional cells in $\Rn^n$, with ridges their
$(n-2)$-dimensional faces. Suppose that their hyperplanes miss
the origin, every ridge is a face of exactly two members, and the cones
from the origin over distinct members have disjoint interiors.
Then the members of $\mathcal F$ meet face-to-face, and radial projection
of $S = \bigcup_{F\in\mathcal F}F$ onto $\Sph^{n-1}$ is a homeomorphism.
\end{proposition}

\begin{proof}
Radial projection from the center carries each facet to a convex spherical cell of
$\Sph^{n-1}$ of dimension $n-1$, and is a homeomorphism on the facet,
since the facet is convex and its hyperplane misses the center. It carries
the ridges to the $(n-2)$-faces of the spherical cells, and the spherical cells of two facets
have disjoint interiors, since the facets do not overlap when seen from
the center.

Every ray meets $S$. Otherwise some direction $y$ lies in no
cell. Take a point $z$ in the interior of a cell and an arc of a great
circle from $z$ to $y$, moved slightly so that it meets no face of
dimension below $n-2$ and is not tangent to the great sphere of any
$(n-2)$-face. The union of the cells is closed, so the arc has a last
point $w$ up to which it stays in the union. Then $w$ is in the interior of
no cell, so it lies in the relative interior of an $(n-2)$-face $R$ of
some cell $C$, and near $w$ the cell $C$ is one closed side of the great
sphere $G$ of $R$. The original ridge projecting to $R$ belongs to
exactly one other facet, whose spherical cell $C'$ is likewise one
closed side of $G$ near $w$. It is the other
side, since on the same side $C$ and $C'$ would share interior points. So
$C'$ contains the points of the arc just past $w$, which contradicts the
choice of $w$.

Each ray meets $S$ at most once. A ray meets a facet at most once, since
it meets the facet's hyperplane at most once. Let two facets $F$ and $F'$
meet the ray in the direction $y$, so that their cells $C$ and $C'$
contain $y$. Take a small sphere $\Sigma$ about $y$ in $\Sph^{n-1}$, of
radius less than the distance from $y$ to every cell and every face that
does not contain $y$. Since $n \ge 3$, $\Sigma$ is connected.
The interiors of $C$ and $C'$ meet $\Sigma$, and a path in $\Sigma$
from one to the other, moved slightly as before, passes
from cell to cell only through relative interiors of $(n-2)$-faces
containing $y$. At such a point the cell on the far side is
the cell of the other facet containing that ridge: it covers the far side
near the point, and any other cell meeting the far side there would
overlap it. So consecutive facets along the path share a ridge whose
radial image contains $y$, and both meet the ray at the point of that
ridge on it. Hence $F$ and
$F'$ meet the ray at the same point.

The same path construction proves that the facets meet face-to-face.
Let $x$ lie in the relative interior of a nonempty intersection $F\cap F'$,
and apply the construction with $y=x/\|x\|$.
Consecutive facets along the path share a ridge containing $x$, so their
minimal faces containing $x$ coincide with the minimal face of that ridge.
This face is therefore the same for $F$ and $F'$. Since $x$ is in the
relative interior of $F\cap F'$, their intersection is exactly that
common face.

$S$ is compact and radial projection is a continuous bijection
from $S$ onto $\Sph^{n-1}$, so it is a homeomorphism.
\end{proof}

\FloatBarrier

\section*{Statements and Declarations}

\paragraph{Data and code availability.} The code and data for
Section~\ref{sec:method} are supplied as supplementary material. They include the 12
configurations, coordinate models for the facet types, facet lists,
stress certificates, and
the $A_7^{+}$ construction and classification certificates. The supplied
records include fresh verification logs and manifests identifying the
source and data by SHA-256 hashes. The command in
Section~\ref{sec:method} reproduces the tables and certificate checks.

\paragraph{Supplementary material.} Exact-arithmetic verification code, configuration
data, facet lists, certificates, and verification records, with instructions
for reproducing the computations.

\paragraph{Use of large language models.} The author used Claude
(Anthropic) and Codex (OpenAI) as coding and writing assistants. They
drafted, reviewed, and refactored verification code, and drafted and
edited text under the author's direction. The author takes full
responsibility for the mathematical claims, verification, and final text.

\paragraph{Competing interests.} The author has no financial or
non-financial interests related to this work.

\paragraph{Funding.} No funding was received for this work.

\end{document}